\documentclass[12pt,reqno]{amsart}
\usepackage{amssymb,latexsym}
\newtheorem{theorem}{Theorem}
\newtheorem{lemma}{Lemma}

\newtheorem*{theoremb}{Theorem B (Wolff)}
\newtheorem*{theorema} {Theorem A (Wolff)}

\begin{document}

\title [Wolff's Problem of Ideals]
{Wolff's Problem of Ideals in the Multipler Algebra on Dirichlet Space}
\author{Debendra P. Banjade and Tavan T. Trent}
\address{Department of Mathematics\\
         The University of Alabama\\
         Box 870350\\
         Tuscaloosa, AL  35487-0350\\
         (205)758-4275}
\email{dpbanjade@crimson.ua.edu, ttrent@as.ua.edu}
\begin{abstract}
We establish an analogue of Wolff's theorem on ideals in $H^{\infty}(\mathbb{D})$ for the multiplier algebra of Dirichlet space.
\end{abstract}
\maketitle

In 1962 Carleson [C] proved his famous \lq\lq Corona theorem" characterizing when a finitely generated ideal in
$H^{\infty}(\mathbb{D})$ is actually all of $H^{\infty}(\mathbb{D})$.  Independently, Rosenblum [R], Tolokonnikov [To], and
Uchiyama gave an infinite version of Carleson's work on $H^{\infty}(\mathbb{D})$.  In an effort to classify ideal membership
for finitely-generated ideals in
$H^{\infty}(\mathbb{D})$, Wolff [G] proved the following version:

\begin{theorema}
If
\begin{align}
\{f_j\}_{j=1}^n \subset H^{\infty}(\mathbb{D}), H \in H^{\infty}(\mathbb{D}) \quad \text{and}\notag\\
|H(z)| \le \left( \sum_{j=1}^n \, |f_j(z)|^2\right)^{\frac 1{2}} \;\; \text{for all } \, z \in \mathbb{D},
\end{align}
then
\[
H^3 \in \mathcal{I} ( \{ f_j\}_{j=1}^n),
\]

\medskip \noindent
the ideal generated by $\{f_j\}_{j=1}^n$ in $H^{\infty}(\mathbb{D})$.
\end{theorema}
It is known that (1) is not, in general, sufficient for $H$ itself to be in $\mathcal{I} ( \{ f_j\}_{j=1}^n)$, see Rao [G];
or even for $H^2$ to be in $\mathcal{I} ( \{ f_j\}_{j=1}^n)$, see Treil [T].

\medskip
Recall that if we consider the radical of the ideal, $\mathcal{I} ( \{ f_j\}_{j=1}^n)$,  i.e.
\[
Rad \, ( \{ f_j\}_{j=1}^n)) \overset{def}{=} \{h \in H^{\infty}(\mathbb{D}): \, \exists \, n \in \mathbb{N} \, \text{ with } \, h^n \in
\mathcal{I} ( \{ f_j\}_{j=1}^n)\},
\]

\medskip \noindent
then (1) gives a characterization of radical ideal membership.

\newpage
That is:
\begin{theoremb}
Let $\{H, f_j: \, j=1, \dots, n\} \subset H^{\infty}(\mathbb{D})$.  Then \, $H \in
Rad \, ( \{ f_j\}_{j=1}^n)$ \, if and only if there exists $C_0 < \infty$ and $ m \in \mathbb{N}$ such that
\[
|H^m(z)| \le C_0   \sum_{j=1}^n \, |f_j(z)|^2 \; \text{ for all } z \in \mathbb{D}.
\]
\end{theoremb}

For the algebra of multipliers on Dirichlet space, the analogue of the corona theorem was established in Tolokonnikov [To] and, for infinitely
many generators, this was done in Trent [Tr2].  The purpose of this paper is to establish an analogue of Wolff's results, Theorems A and B, for the algebra of multipliers on Dirichlet space.

\medskip
We use $\mathcal{D}$ to denote the Dirichlet space on the unit disk, $\mathbb{D}$. That is,
\begin{multline*}
\mathcal{D} =  \{ \, f:\; \mathbb{D}\rightarrow\mathbb{C}\mid \; f \text{ is analytic on } \mathbb{D} \text{ and for }
 f(z)=\overset{\infty}{\underset{n=0}{\sum}}a_{n}\, z^n,\\
\left\Vert f \right\Vert_{\mathcal{D}}^{2}
=\underset{n=0}{\overset{\infty}{\sum}}\, (n+1)\left| a_{n}\right|^{2}< \infty \}.
\end{multline*}

\medskip
We will use other equivalent norms for smooth functions in $\mathcal{D}$
as follows,
\[
\left\Vert f\right\Vert _{\mathcal{D}}^{2}=\int_{-\pi}^{\pi}\vert f\vert^{2}d\sigma+\int_{D}\vert f\,'(z)\vert^{2}\, dA(z)\qquad \mbox{ and}
\]
\[
\left\Vert f \right\Vert _{\mathcal{D}}^{2}=\int_{-\pi}^{\pi}\vert f\vert^{2}d\sigma +
\int_{-\pi}^{\pi}\int_{-\pi}^{\pi}\frac{\vert f(e^{it})-f(e^{i\theta})\vert^{2}}{\vert e^{it}-e^{i\theta}\vert^{2}}\, d\sigma d\sigma.
\]

Also, we will consider $\overset{\infty}{\underset{1}{\oplus}}\, \mathcal{D}$
as an $l^{2}$-valued Dirichlet space. The norms in this case are exactly
as above but we will replace the absolute value by $l^{2}$-norms.
Moreover, we use $\mathcal{HD}$ to denote the harmonic Dirichlet
space (restricted to the boundary of $\mathbb{D}$). The functions in $\mathcal{D}$
have only vanishing negative Fourier coefficients, whereas the functions in $\mathcal{HD}$
may have negative fourier coefficients which do not vanish. Again, if $f$ is smooth
on $\partial D$, the boundary of the unit disk $D$, then
\[
\left\Vert f \right\Vert _{\mathcal{HD}}^{2}=\int_{-\pi}^{\pi}\vert f\vert^{2}\, d\sigma +
\int_{-\pi}^{\pi}\int_{-\pi}^{\pi}\frac{\vert f(e^{it})-f(e^{i\theta})\vert^{2}}{\vert e^{it}-e^{i\theta}\vert^{2}}\, d\sigma d\sigma.
\]

We use $\mathcal{M}(\mathcal{D})$ to denote the multiplier
algebra of Dirichlet space, defined as:
$\mathcal{M}(\mathcal{D})=\left\{ \phi\in\mathcal{D}:\,\phi f\in\mathcal{D}\; \text{ for all } f\in\mathcal{D}\right\},$
and we will denote the multiplier algebra of harmonic Dirichlet space
by $\mathcal{M}(\mathcal{HD})$, defined similarly (but only on $\partial D$).

Given $\left\{ f_{j}\right\} _{j=1}^{\infty}\subset\mathcal{M}(\mathcal{D})$,
we consider $F(z)=\left(f_{1}(z),\, f_{2}(z),\dots \right)$ for $z\in D.$
We define the row operator $M_{F}^{R}: \overset{\infty}{\underset{1}{\oplus}}\, \mathcal{D\rightarrow\mathcal{D}}$ \,
by
\[
M_{F}^{R}\left(\left\{ h_{j}\right\} _{j=1}^{\infty}\right)=\overset{\infty}{\underset{j=1}{\sum}}\,f_{j}h_{j} \,
\mbox{ for } \left\{ h_{j}\right\} _{j=1}^{\infty}\in\overset{\infty}{\underset{1}{\oplus}}\, \mathcal{D}.
\]
Similarly, define the column operator $M_{F}^{C}:\mathcal{D}\rightarrow\overset{\infty}{\underset{1}{\oplus}}\, \mathcal{D}$\,
by
\[
M_{F}^{C}\left(h\right)=\left\{ f_{j}h\right\} _{j=1}^{\infty}
\, \mbox{ for } \,h\in\mathcal{D}.
\]

\medskip
We notice that $\mathcal{D}$ is a reproducing kernel (r.k.) Hilbert
space with r.k.
\[
k_{w}(z)=\frac{1}{z\overline{w}}\, log\left(\frac{1}{1-z\overline{w}}\right)
\, \mbox{ for } z,\, w\in \mathbb{D}
\]
 and it is well known (see [AM]) that
\[
\frac{1}{k_{w}(z)}=1-\overset{\infty}{\underset{n=1}{\sum}}\,c_{n}(z\overline{w})^{n},\, c_{n}>0,
\mbox { for all } n.
\]
 Hence, Dirichlet space has a reproducing kernel
with {}``one positive square'' or a {}``complete Nevanlinna-Pick''
kernel. This property will be used to complete the first part of our proof.

\medskip
 An important relationship between the multipliers and reproducing
kernels is that for $\phi\in\mathcal{M}(\mathcal{D})$ \, and  ${z\in \mathbb{D}}$,
\[
M_{\phi}^{\star}k_{z}=\overline{\phi(z)}\; k_{z}.
\]
This automatically
implies that $\Vert\phi\Vert_{\infty}\leq\Vert M_{\phi}\Vert,$
so $\mathcal{M}(\mathcal{D})\subseteq H^{\infty}(\mathbb{D}).$

Similarly, if $\phi_{ij}\in\mathcal{M}(\mathcal{D})$ and $M_{[\phi_{ij}]_{j=1}^{\infty}}\in B\,(\overset{\infty}{\underset{1}{\oplus}}\, \mathcal{D}),$
then for $\underline{x}\in l^{2}$ and $z\in \mathbb{D},$ we have
\[
M_{[\phi_{ij}]}^{\star}\left(\underline{x}\, k_{z}\right)=[\phi_{ij}(z)]^{\star}\;\left(\underline{x}\, k_{z}\right).
\]
Again, it follows that
\[
\underset{z\in D}{sup}\,\Vert\,[\phi_{ij}(z)]\,\Vert_{B(l^{2})}\leq\Vert M_{[\phi_{ij}]}\Vert_{B\,(\overset{\infty}{\underset{1}{\oplus}}\, \mathcal{D})}
\]
and so
\[
\mathcal{M}(\overset{\infty}{\underset{1}{\oplus}}\, \mathcal{D})\subseteq {H_{B(l^{2})}^{\infty}(\mathbb{D}).}
\]

 It is clear that $\mathcal{M}(H^{2}(\mathbb{D}))=H^{\infty}(\mathbb{D})$ but $\mathcal{M}(\mathcal{D})\varsubsetneq H^{\infty}(\mathbb{D})$
(e.g., $\overset{\infty}{\underset{n=1}{\sum}}\frac{z^{n^{3}}}{n^{2}}$
is in $H^{\infty}(\mathbb{D})$ but is not in $\mathcal{D}$ and so not
in $\mathcal{M}(\mathcal{D})$).
Hence, $\mathcal{M}(\mathcal{D})\subsetneq H^{\infty}(\mathbb{D})\,\cap\,\mathcal{D}$.

\medskip
 Also, it is worthwhile to note that the pointwise hypothesis that

\medskip \noindent
$F(z)\, F(z)^{\star}$ $\leq 1$
for $z\in \mathbb{D}$, implies that the analytic Toeplitz operators $T_{F}^{R}$
and $T_{F}^{C}$ defined on $\overset{\infty}{\underset{1}{\oplus}}\, H^{2}(\mathbb{D})$
and $H^{2}(\mathbb{D})$ in analogy to that of $M_{F}^{R}$ and $M_{F}^{C}$
are bounded and
\[
 \left\Vert T_{F}^{R}\right\Vert =\left\Vert T_{F}^{C}\right\Vert =\underset{z\in \mathbb{D}}{sup}\left(\overset{\infty}{\underset{j=1}{\sum}}\,\vert f_{j}(z)\vert^{2}\right)^{\frac{1}{2}}\le 1.
\]

But, since $M(\mathcal{D})\varsubsetneq H^{\infty}(\mathbb{D})$, the pointwise
upperbound hypothesis will not be sufficient to conclude that $M_{F}^{R}$
and $M_{F}^{C}$ are bounded on Dirichlet space. However,  $\left\Vert M_{F}^{R}\right\Vert \leq\sqrt{18}\left\Vert M_{F}^{C}\right\Vert$ from [Tr2].
 Thus, we will replace the natural normalization that $F(z)\,F(z)^{\star}\leq 1$
for all $z\in \mathbb{D}$, by the stronger condition that $\left\Vert M_{F}^{C}\right\Vert \leq 1.$

\medskip
Then we have the following theorem:
\begin{theorem}
Let $H$, $\{f_j\}_{j=1}^{\infty} \subset \mathcal{M}(\mathcal{D})$.  Assume that
\begin{align*}
& (\mathrm{a}) \; \Vert M_F^C \Vert \le 1 \\
\text{and } \; \; & (\mathrm{b}) \; |H(z)| \le \sqrt{\sum_{j=1}^{\infty} \, |f_j(z)|^2} \; \text{ for all } z \in \mathbb{D}.
\end{align*}

Then there exist $\{g_j\}_{j=1}^{\infty} \subset \mathcal{M}(\mathcal{D})$ with
\begin{align*}
& \; \Vert M_G^C \Vert < \infty\\
\text{and } \;\; &  \; F \, G^T = H^3.
\end{align*}
\end{theorem}

\medskip
Of course, it should be noted that for only a finite number of multipliers, $\{f_j\}$, condition (a) of Theorem 1 can always be assumed, so we have the exact analogue of Wolff's theorem in the finite case.

\medskip
First, let's outline the method of our proof. Assume that $F\in\mathcal{M}_{l^{2}}(\mathcal{D})$
and $H\in\mathcal{M}(\mathcal{D})$ satisfy the hypotheses (a) and (b) of Theorem 1.
Then we show that there exists a constant
$K<\infty,$ so that
\begin{equation}
M_{H^3}\, M_{H^3}^{\star}\leq K^{2}M_{F}^{R}\, M_{F}^{\star R}.
\end{equation}

Given $(2)$, a commutant lifting theorem argument as it appears in,
for example, Trent {[}Tr2{]}, completes the proof by providing
a $G\in \mathcal{M}_{l^{2}}(\mathcal{D})$, so that $\Vert M_{G}^{C}\Vert\leq K$ and $F\, G^{T}=H^3$.\\

 But (2) is equivalent to the following: there exists a constant
$K<\infty$ so that, for any $h\in\mathcal{D}$, there exists \,$\underline{u}_{h}\in\overset{\infty}{\underset{1}{\oplus}}\, \mathcal{D}$\,
such that
\begin{align}
(\mathrm{i})&\,\,\, M_{F}^{R}(\underline{u}_{h})=  H^3h \quad \text{ and}\notag \\
(\mathrm{ii}) &\,\,\left\Vert \underline{u}_{h}\right\Vert _{\mathcal{D}}\leq K\left\Vert h\right\Vert _{\mathcal{D}}.
\end{align}

 \medskip
 Hence, our goal is to show that (3) follows from (a) and (b).
For this we need a series of lemmas.

\begin{lemma}
 Let $\left\{c_{j}\right\}_{j=1}^{\infty}\in l^{2}$ and $C=\left(c_{1},c_{2},...\right)\in B\left(l^{2},\mathbb{C}\right).$
Then there exists $Q$ such that the entries of $Q$ are either $0$ or
$\pm c_{j}$ for some $j$ and $CC^{\star}I-C^{\star}C=QQ^{\star}.$
Also, range of $Q$ $=$ kernel of $C.$
\end{lemma}
 We will apply this lemma in
our case with $C=F(z)$ for each $z\in \mathbb{D}$, when $F(z)\neq 0$. A proof of a more general version can be found in Trent {[}Tr2{]}.

\medskip
Given condition
(b) of Theorem 1 for all $z\in \mathbb{D}$, $F\in\mathcal{M}_{l^{2}}(\mathcal{D})$
and $H\in\mathcal{M}(\mathcal{D})$ with $H$ being not identically
zero, we lose no generality assuming that $H(0)\neq 0.$ If $H(0)=0,$
but $H(a)\neq 0,$ let $\beta(z)=\frac{a-z}{1-\bar{a}\, z}$ for $z\in \mathbb{D}$.
 Then since (b) holds for all $z\in \mathbb{D}$, it holds for $\beta(z)$.
So we may replace $H$ and $F$ by $Ho\beta$ and $Fo\beta,$ respectively.
If we prove our theorem for $Ho\beta$ and $Fo\beta$, then there
exists $G\in\mathcal{M}_{l^{2}}(\mathcal{D}$) so that $\left(Fo\beta\right)G=Ho\beta$
and hence $F(Go\beta^{-1})=H\,$ and $Go\beta^{-1}\in\mathcal{M}_{l^{2}}(\mathcal{D}),$
so we were done. Thus, we may assume that $H(0)\neq0$ in (b), so
$\Vert F(0)\Vert_{2}\neq 0.$  This normalization will let us apply some relevant lemmas from [Tr1].

\medskip
It suffices to establish (i) and (ii) for any dense set of
functions in $\mathcal{D}$, so we will use polynomials. First, we will assume $F$ and $H$ are analytic
on $\mathbb{D}_{1+\epsilon}(0)$. In this case, we write the most general solution of the pointwise problem on $\bar{\mathbb{D}}$ and find an analytic solution with uniform bounds. Then we remove the smoothness hypotheses on $F$ and $H$.

\medskip
For a polynomial, $h$, we take
\[
\underline{u}_{h}(z)=F(z)^{\star}\left(F(z)F(z)^{\star}\right)^{-1}H^3\, h - Q(z)\underline{k}(z),\mbox{ where }\underline{k}(z)\in l^{2}\mbox{ for }z\in \bar{\mathbb{D}}.
\]
 We have to find $\underline{k}(z)$ so that $\underline{u}_{h}\in\overset{\infty}{\underset{1}{\oplus}}\, \mathcal{D}$.
Thus we want $\bar{\partial_{z}}\,\underline{u}_{h}=0$ in $\mathbb{D}.$

\medskip
Therefore, we will try
\[
\underline{u}_{h}=\frac{F^{\star}H^3 h}{FF^{\star}}-Q\,\,\widehat{\left(\frac{Q^{\star}F^{'\star}H^3h}{\left(FF^{\star}\right)^{2}}\right)},
\]
where $\widehat{k}$ is the Cauchy transform of $k$ on $\mathbb{D}$. Note that
for $k$ smooth on $\bar{\mathbb{D}}$ and $z\in \mathbb{D}$,
\[
\underline{\widehat{k}}(z)=-\frac{1}{\pi}\int_{D}\frac{\underline{k}(w)}{w-z}\, dA(w) \, \text{ and }
\overline{\partial} \, \underline{\widehat{k}}(z) = k(z) \, \text{ for } z \in \mathbb{D}.
\]
See [A] for background on the Cauchy transform.

Then it's clear that $M_{F}^{R}\left(\underline{u}_{h}\right)=H^3h$
and $\underline{u}_{h}$ is analytic. Hence, we will be done in the
smooth case if we are able to find $K<\infty$, independent of the polynomial, $h$, and $\epsilon >0$,
such that
\begin{equation}
\left\Vert \underline{u}_{h}\right\Vert_{\mathcal{D}}\leq K\left\Vert h\right\Vert_{\mathcal{D}}
\end{equation}

\begin{lemma}
Let $\underline{w}$ be a harmonic function on $\overline{\mathbb{D}}$,
then
\[
\int_{D}\Vert Q'\underline{w}\Vert_{l^{2}}^{2}\, dA \leq 8 \,\Vert\underline{w}\Vert_{\mathcal{HD}}^{2}.
\]
\end{lemma}

\begin{proof}
Let $\underline{w}$ be a vector-valued harmonic function
on $\overline{\mathbb{D}}$. Write $\underline{w}=\underline{x}+\underline{\bar{y}}$,
where $\underline{x}$ and $\underline{\bar{y}}$ are respectively
the analytic and co-analytic parts of $\underline{w}$.

\medskip
We have
\begin{align*}
\int_{D}\Vert Q'\underline{w}\Vert_{l^{2}}^{2}\, dA & =\int_{D}\Vert Q'\underline{x}+Q'\underline{\bar{y}}\Vert_{l^{2}}^{2}\,dA\\
& \leq2\,\int_{D}\Vert Q'\underline{x}\Vert_{l^{2}}^{2}\,dA+2\,\int_{D}\Vert Q'\underline{\bar{y}}\Vert_{l^{2}}^{2}\,dA.
\end{align*}
 Now
 \begin{align*}
 \int_{D}\Vert Q'\underline{x}\Vert_{l^{2}}^{2}\,dA & =\int_{D}<Q'^{\star}Q'\underline{x}\,,\,\underline{x}>_{l^{2}}dA\\
 & \leq\int_{D}<F'F'^{\star}\underline{x},\,\underline{x}>_{l^{2}}dA\\
 & \leq\int_{D}\overset{\infty}{\underset{j=1}{\sum}}\,\overset{\infty}{\underset{k=1}{\sum}}\vert\bar{f_{j}'}x_{k}\vert^{2}\,dA \\
& \leq2\,\overset{\infty}{\underset{j=1}{\sum}}\,\overset{\infty}{\underset{k=1}{\sum}}\int_{D}\vert(f_{j}x_{k})'\vert^{2}\, dA+2\,\overset{\infty}{\underset{j=1}{\sum}}\overset{\infty}{\underset{k=1}{\sum}}\int_{D}\vert f_{j}x_{k}'\vert^{2}\,dA\\
& \leq2\,\overset{\infty}{\underset{j=1}{\sum}}\,\overset{\infty}{\underset{k=1}{\sum}}\Vert M_{f_{j}}x_{k}\Vert_{\mathcal{D}}^{2}+2\,\overset{\infty}{\underset{k=1}{\sum}}\Vert x_{k}\Vert_{\mathcal{D}}^{2}\,dA\\
& \leq2\,\overset{\infty}{\underset{k=1}{\sum}}\Vert M_{F}^{C}\Vert_{\mathcal{D}}^{2}\,\Vert x_{k}\Vert_{\mathcal{D}}^{2}+2\,\overset{\infty}{\underset{k=1}{\sum}}\Vert x_{k}\Vert_{\mathcal{D}}^{2}\\
& =4\,\Vert\underline{x}\Vert_{\mathcal{D}}^{2}.
\end{align*}
 Similarly, we can show that $\int_{D}\Vert Q'\underline{\bar{y}}\Vert_{l^{2}}^{2}\,dA\leq4\,\Vert\underline{y}\Vert_{\mathcal{D}}^{2}$.

 \medskip
 Thus,
 \begin{align*}
 \int_{D}\Vert Q'\underline{w}\Vert_{l^{2}}^{2}\, dA & \leq8\,\Vert\underline{x}\Vert_{\mathcal{D}}^{2}+8\,\Vert\underline{y}\Vert_{\mathcal{D}}^{2}\\
 & =8\,\Vert\underline{x}+\underline{\bar{y}}\Vert_{\mathcal{H}\mathcal{D}}^{2}\\
 & =8\,\Vert\underline{w}\Vert_{\mathcal{HD}}^{2}.
 \end{align*}
 \vskip-1em
 \end{proof}

\begin{lemma}
 Let the operator $T$ be defined on $L^2(\mathbb {D},dA)$ by
\[
  (Tf)(\lambda)=\int_{D}\frac{f(z)}{\left(z-\lambda\right)\left(1-z\,\bar{\lambda}\right)}\,dA(z),
\]
for $\lambda\in \mathbb{D}$ and $f \in L^2(\mathbb {D},dA)$.
Then
\[
 \Vert Tf\Vert_{A}^{2}\leq100\,\pi^{2}\,\Vert f\Vert_{A}^{2}.
 \]
\end{lemma}

\begin{proof}
To show that the singular integral operator, $T$, is bounded on $L^2(\mathbb {D},dA)$, we apply Zygmund's method of rotations [Z] and apply Schur's lemma an infinite number of times.

\medskip
Let $f(z)=\overset{\infty}{\underset{j=0}{\sum}}\,\overset{\infty}{\underset{k=0}{\sum}}\,a_{jk}z^{j}\bar{z}^{k}$,
where $a_{ij}=0$ except for a finite number of terms.  For $z = r \, e^{i\theta}$, we relabel, so that

\[
 f(r \,e^{i\theta})
 =\overset{\infty}{\underset{l=-\infty}{\sum}}f_{l}(r)\,e^{il\theta},
\text{ where }\, f_{l}(r)=\overset{\infty}{\underset{k=0}{\sum}}a_{l+k\, k\,}r^{l+2k}.
\]
 Then
 \[
 \Vert f\Vert_{A}^{2}=\overset{\infty}{\underset{l=-\infty}{\sum}}\Vert f_{l}(r)\Vert_{L^{2}\left[0,1\right]},
 \]
where the measure on $L^{2}\left[0,1\right]$ is {}``$rdr$''.

\medskip
 Now
 \begin{align*}
 (Tf)(\lambda)& =\int_{D}\frac{f(z)}{\left(z-\lambda\right)\left(1-z\bar{\lambda}\right)}\,dA(z)\\
& =\overset{\infty}{\underset{n=0}{\sum}}\,\overline{\lambda}^{\, n}\,\int_{D}\,\left[\frac{1}{z-\lambda}\right]z^{n}f(z)\, dA(z)\\
& =\overset{\infty}{\underset{n=0}{\sum}}\,\overline{\lambda}^{\, n}\,\left[\int_{\left|z\right|<\left|\lambda\right|}\frac{1}{z-\lambda}+\int_{\left|\lambda\right|<\left|z\right|}\frac{1}{z-\lambda}\right]z^{n}f(z)\, dA(z)\\
& =\overset{\infty}{\underset{n=0}{\sum}}\overline{\lambda}^{\, n}\left[\frac{1}{-\lambda}\int_{\left|z\right|<\left|\lambda\right|}\overset{\infty}{\underset{p=0}{\sum}}\frac{z^{p}}{\lambda^{p}}+\int_{\left|\lambda\right|<\left|z\right|}
\frac{1}{z}\overset{\infty}{\underset{p=0}{\sum}}\frac{\lambda^{p}}{z^{p}}\right]z^{n}f(z)\,  dA(z)\\
& =\overset{\infty}{\underset{n=0}{\sum}}\,\overset{\infty}{\underset{p=0}{\sum}}(-1)\,\overline{\lambda}^{\, n}\,\frac{1}{\lambda}\int_{\left|z\right|<\left|\lambda\right|}\frac{z^{n+p}}{\lambda^{p}}\left(\overset{\infty}{\underset{l=-\infty}{\sum}}f_{l}(r)e^{il\theta}\right)\, dA(z)\\
& \quad +\overset{\infty}{\underset{n=0}{\sum}}\,\overline{\lambda}^{n}\overset{\infty}{\underset{p=0}{\sum}}\int_{\left|\lambda\right|<\left|z\right|}\,
\frac{\lambda^{p}}{z^{p+1}}z^{n}\left(\overset{\infty}{\underset{l=-\infty}{\sum}}f_{l}(r)e^{il\theta}\right)\,dA(z).
\end{align*}
 Therefore,
 \begin{align*}
 & \left(Tf\right)(se^{it}) = \\
 & \quad \overset{\infty}{\underset{l=-\infty}{\sum}}\overset{\infty}{\underset{n=0}{\sum}}\overset{\infty}{\underset{p=0}{\sum}}(-1)\,s^{n}e^{-int}
 \frac{e^{-it}}{s}\int_{-\pi}^{\pi}\int_{0}^{s}\frac{r^{n+p}e^{i(n+p+l)\theta}}{s^{p}e^{ipt}}\,f_{l}(r) r dr d\theta
 \end{align*}
 \begin{equation*}
 \quad \;\, +\overset{\infty}{\underset{l=-\infty}{\sum}}\overset{\infty}{\underset{n=0}{\sum}}s^{n}e^{-int}\,\overset{\infty}{\underset{p=0}{\sum}}\int_{-\pi}^{\pi}\int_{s}^{1}
\frac{s^{p}e^{ipt}}{r^{p+1}e^{i(p+1)\theta}}r^{n}e^{in\theta}e^{il\theta}f_{l}(r)\, r dr\, d\theta.\tag*{($\star$)}
\end{equation*}

\medskip
 Taking $l=0$ in  $(\star)$, we get that
\begin{align*}
 \left(Tf_{0}\right)(se^{it})& =\overset{\infty}{\underset{n=0}{\sum}}\,\overset{\infty}{\underset{p=0}{\sum}}(-1)s^{n}e^{-int}\,\frac{e^{-it}}{s}\int_{-\pi}^{\pi}\int_{0}^{s}
 \frac{r^{n+p}e^{i(n+p)\theta}}{s^{p}e^{ipt}}\,f_{0}(r)\, r dr\, d\theta\\
& \;\; +\overset{\infty}{\underset{n=0}{\sum}}s^{n}e^{-int}\overset{\infty}{\underset{p=0}{\sum}}\int_{-\pi}^{\pi}\int_{s}^{1}\,\frac{s^{p}e^{ipt}}{r^{p+1}e^{i(p+1)\theta}}r^{n}e^{in\theta}f_{l}\,(r) rdr\,d\theta.
\end{align*}

 \medskip
 Simplifying the above,
 \begin{align*}
 \left(Tf_{0}\right)(se^{it})& =-2\pi\int_{0}^{1}\chi_{\left(0,s\right)}(r)\,\frac{f_{0}(r)\, e^{-it}}{s}\: rdr\\
& \quad +2\pi s\, e^{-it}\overset{\infty}{\underset{n=0}{\sum}}s^{2n}\int_{0}^{1}\chi_{(s,1)}(r)\, f_{0}(r)\, rdr.
\end{align*}
 So
 \[
 \left(Tf_{0}\right)(se^{it})=2\,\pi e^{-it}\left(T_{0}f_{0}\right)(s),
\]
where we define $T_0$ on $L^2([0,1],rdr)$ by
\[
\left(T_{0}f_{0}\right)(s)=-\int_{0}^{1}\chi_{\left(0,s\right)}(r)\left(\frac{r}{s}\right)f_{0}(r)\,dr
+\frac{s}{1-s^{2}}\int_{0}^{1}\chi_{(s,1)}(r)\, f_{0}(r)\, rdr.
\]

A similar calculation shows that
when $l\geq1,$ then
\[
\left(Tf_{l}(r)\,e^{il\theta}\right)(se^{it})=2\pi e^{i(l-1)t}\left(T_{l}f_{l}\right)(s),
\]
where
\[
\left(T_{l}f_{l}\right)(s)=\frac{1}{1-s^{2}}\int_{0}^{1}\chi_{(s,1)}(r)\left(\frac{s}{r}\right)^{l-1}f_{0}(r)\, rdr.
\]
Similarly, when $l<0,$
\[
\left(Tf_{l}(r)e^{il\theta}\right)(se^{it})=2\pi e^{i(l-1)t}\left(T_{l}f_{l}\right)(s),
\]
where
\begin{align*}
\left(T_{l}f_{l}\right)(s) =& -\left(\underset{n=0}{\overset{-l}{\sum}}\,s^{2n}\right)\int_{0}^{1}\chi_{\left(0,s\right)}(r)\left(\frac{r}{s}\right)^{1-l}f_{l}(r)dr\\
& +\frac{1}{1-s^{2}}\int_{0}^{1}\chi_{(s,1)}(r)\,\left(rs\right)^{1-l}f_{l}(r)dr.
\end{align*}
Hence,
\[
\left(Tf\right)\left(se^{it}\right)=2\pi\underset{l=-\infty}{\overset{\infty}{\sum}}e^{i(l-1)t}\left(T_{l}f_{l}\right)(s),
\]
\[
\mbox{for \ensuremath{\left(T_{l}f_{l}\right)(s)=\begin{cases}
\begin{array}{l}
-(\underset{n=0}{\overset{-l}{\sum}}s^{2n})\int_{0}^{1}\chi_{\left(0,s\right)}(r)\left(\frac{r}{s}\right)^{1-l}f_{l}(r)\,dr\\
\quad +\frac{1}{1-s^{2}}\int_{0}^{1}\chi_{(s,1)}(r)\,\left(rs\right)^{1-l}f_{l}(r)\,dr\quad \text{for }\, l\leq0
\end{array}\\
\; \frac{1}{1-s^{2}}\int_{0}^{1}\chi_{(s,1)}(r)\left(\frac{s}{r}\right)^{l-1}f_{0}(r)\, rdr\quad\quad\;\; \text{for }\, l>0.
\end{cases}}}
\]

\medskip
By our construction,
\[
\Vert Tf\Vert_{A}^{2}=4\,\pi^{2}\underset{l=-\infty}{\overset{\infty}{\sum}}\vert\vert T_{l}f_{l}\vert\vert_{L^{2}\left[0,1\right]}^{2},
\]
where the measure on $L^{2}\left[0,1\right]$ is {}``$rdr$''.
 Thus to prove our lemma, it suffices to prove that
\begin{equation}
 \underset{l}{sup}\,\Vert T_{l}\Vert_{B\left(L^{2}\left[0,1\right],\, L^{2}\left[0,1\right]\right)}\leq5<\infty.\tag{${\star\star}$}
\end{equation}
Once we prove $(\star\star)$, we can conclude that
\[
\Vert Tf\Vert_{A}^{2}\leq100\,\pi^{2}\underset{l=-\infty}{\overset{\infty}{\sum}}\Vert f_{l}\Vert_{L^{2}\left[0,1\right]}^{2}=100\,\pi^{2}\Vert f\Vert_{A}^{2}.
\]

\medskip
 For the case $l=0$, we get that

 \begin{align*}
 & \int_{0}^{1}\left|\left(T_{0}f_{0}\right)(se^{it})\right|^{2}sds\\
 & \;\;\leq  2\int_{0}^{1}\left|-\int_{0}^{1}\chi_{\left(0,s\right)}(r)\,\frac{f_{0}(r)\,-e^{-it}}{s}\: rdr\right|^{2}sds\\
& \quad\; +2\int_{0}^{1}\left|\frac{s}{1-s^{2}}\,(e^{-it})\int_{0}^{1}\chi_{(s,1)}(r)\, f_{0}(r)\, rdr\right|^{2}sds\\
&\;\;\leq  2\int_{0}^{1}\frac{1}{s^{2}}\left[\int_{0}^{1}\chi_{(0,s)}(u)\,\vert f_{0}(u)\vert\, u\, du\,\int_{0}^{1}\chi_{(0,s)}(v)\,\vert f_{0}(v)\vert\, vdv\right]sds\\
& \quad\; +2\int_{0}^{1}\frac{s^{2}}{\left(1-s^{2}\right)^{2}}\left[\int_{0}^{1}\chi_{(s,1)}(x)\vert f_{0}(x)\vert xdx\int_{0}^{1}\chi_{(s,1)}(y)\vert f_{0}(y)\vert ydy\right]sds.
\end{align*}

\medskip
 Let's take the first term,
which is
\begin{align*}
\int_{0}^{1}\int_{0}^{1} & \vert f_{0}(u)\vert\,\vert f_{0}(v)\vert\left(\int_{0}^{1}\chi_{\left(0,s\right)}(u)\,\chi_{\left(0,s\right)}(v)\,\frac{ds}{s}\right)udu\, vdv\\
& =\int_{0}^{1}\int_{0}^{1}\vert f_{0}(u)\vert\,\vert f_{0}(v)\vert\, ln\left(\frac{1}{max\left\{ u,v\right\} }\right)udu\, vdv.
\end{align*}
 By Schur's Test we know that
 \[
 \int_{0}^{1}\int_{0}^{1}\vert f_{0}(u)\vert\,\vert f_{0}(v)\vert\, ln\left(\frac{1}{max\left\{ u,v\right\} }\right)udu\, vdv\leq C_{1}^2\int_{0}^{1}\left|f(u)\right|^{2}u\, du
\]
if and only if there exist $p\geq0$~~a.e. and $C_{1}<\infty$,
satisfying 

\[\int_{0}^{1}ln\left(\frac{1}{max\left\{ u,v\right\} }\right)p(u)\, udu\,\leq C_{1}\, p(v)\]
for $a.e.$ $v\in\left[0,1\right].$

\medskip
 We take $p(u)=1$.  Now
\[
\int_{0}^{v}\,ln\left(\frac{1}{v}\right)\, u\, du=\frac{1}{4}\,\frac{ln\left(\frac{1}{v^{2}}\right)}{\left(\frac{1}{v^{2}}\right)}\leq\frac{1}{4}
\]
 and
\[
 \int_{v}^{1}\,ln\left(\frac{1}{u}\right)u\, du\leq1-v\leq1.
 \]
Therefore,
\[
C_{1}=\frac{1}{4}+1=\frac{5}{4}.
\]

 Again, we will repeat the same argument for the second term,
 \[
\int_{0}^{1}\int_{0}^{1}\vert f_{0}(x)\vert\,\vert f_{0}(y)\vert\left[\int_{0}^{min\left\{ x,y\right\} }\frac{s^{2}}{(1-s^{2})^{2}\,}sds\right]xdx\, ydy.
\]

 Changing variables and computing, we get that
\begin{align*}
 & \int_{0}^{1}\int_{0}^{1}\vert f_{0}(x)\vert\,\vert f_{0}(y)\vert\left[\int_{0}^{min\left\{ x,y\right\} }\frac{s^{2}}{(1-s^{2})^{2}}\,s\,ds\right]xdx\, ydy\\
& \;\leq\int_{0}^{1}\int_{0}^{1}\vert f_{0}(x)\vert\,\vert f_{0}(y)\vert\left[\frac{1}{2}\frac{1}{\left(1-\left(min\left\{ x,y\right\} \right)^{2}\right)}\right]xdx\, ydy.
\end{align*}

\medskip
 For this second term, we take $p(x)=\frac{1}{\sqrt{1-x^{2}}}.$  Therefore,
\[
 \int_{0}^{y}\frac{1}{2}\,\frac{1}{1-x^{2}}\,\frac{1}{\sqrt{1-x^{2}}}\,xdx
  \leq\frac{1}{2}\,\frac{1}{\sqrt{1-y^{2}}}.
\]
 Also,
\[
\int_{y}^{1}\frac{1}{2}\,\frac{1}{1-y^{2}}\,\frac{1}{\sqrt{1-x^{2}}}\,xdx 
 =\frac{1}{2}\frac{1}{\sqrt{1-y^{2}}}.
\]

\medskip
 So we get\, $C_{2}\leq1.$
Hence
\[
\int_{0}^{1}\left|\left(T_{0}f_{0}\right)(s)\right|^{2}sds  \leq \frac{9}{2}\int_{0}^{1}\left|f_0(u)\right|^{2}u du.
\]

\medskip
Applying Schur's Test for $l\geq1$ with $p(u)=\frac{1}{\sqrt{1-u^{2}}}$,
we get the estimate $C_{l}\leq\frac{3}{2}$, independent of $l$.
 Similarly, for $l<0$ with $p(u)=1$ and $p(u)=\frac{1}{\sqrt{1-u^{2}}}$
for each of the two terms, respectively,
we get the estimate $C_{l}\leq5$, independent of $l$ .
 Thus we conclude that
\[
\underset{l}{sup}\,\Vert T_{l}\Vert_{B\left(L^{2}\left[0,1\right],\, L^{2}\left[0,1\right]\right)}\leq5.
\]
This finishes the proof of Lemma 3.
\end{proof}

\bigskip
\begin{lemma}
If \,$Q$ is a multiplier of $\mathcal{D}$, then
\[
\left(1-\vert z\vert^{2}\right)\vert\, Q'(z)\vert\leq\Vert M_{Q}\Vert_{B(\mathcal{D})} \text{ for all } z \in \mathbb{D}.
\]
\end{lemma}
\begin{proof}
Define\; $\varphi:D\rightarrow D$
as $\varphi(z)=\frac{Q(z)}{\Vert M_{Q}\Vert}$ for all\, $z\in \mathbb{D}$. Now use the Schwarz lemma and the fact that
$\Vert \varphi \Vert_{\infty,\mathbb{D}} \leq \Vert M_{\varphi} \Vert$ to complete the proof.
\end{proof}

\newpage
 We are now ready to prove Theorem 1.

\begin{proof}  First, we will prove the theorem for smooth functions
on $\overline{\mathbb{D}}$ and get a uniform bound. Then we will remove
the smoothness hypothesis.

\medskip
Assume that (a) and (b) of Theorem 1 hold for $F$ and $H$ and that $F$ and $H$ are analytic on $\mathbb{D}_{1+ \epsilon}(0)$.

\medskip
Then our main goal is to show that there exists a constant $K<\infty$, independent of $\epsilon$, so that for any polynomial,  $h$,
there exists $\underline{u}_{h}\in\overset{\infty}{\underset{1}{\oplus}}\,\mathcal{D}$
such that
$M_{F}^{R}(\underline{u}_{h})=H^{3}h$ and $\Vert\underline{u}_{h}\Vert_{\mathcal{D}}^{2}\leq K\,\Vert h\Vert_{\mathcal{D}}^{2}$.

\medskip
We  take $\underline{u}_{h}=\frac{F^{\star}H^{3}h}{FF^{\star}}-Q\widehat{\left(\frac{Q^{\star}F'^{\star}H^{3}h}{\left(FF^{\star}\right)^{2}}\right)}$.
Then  $\underline{u}_{h}$ is analytic and $M_{F}^{R}(\underline{u}_{h})=H^{3}h.$

\medskip
We know that
\[
\left\Vert \underline{u}_{h}\right\Vert _{\mathcal{D}}^{2}=\int_{-\pi}^{\pi}\Vert\underline{u}_{h}(e^{it})\Vert^{2}\,d\sigma(t)+\int_{D}\Vert\left(\underline{u}_{h}(z)\right)^{'}\Vert^{2}\,dA(z).
\]
Condition (b) implies that
\[
\int_{-\pi}^{\pi}\Vert\frac{F^{\star}H^{3}h}{FF^{\star}}-Q\,\widehat{\left(\frac{Q^{\star}F'^{\star}H^{3}h}{\left(FF^{\star}\right)^{2}}\right)}\Vert^{2}\,d\sigma(t)\leq C_{0}^2\Vert h\Vert_{\sigma}^{2},
 \]
where $C_{0}$ can be chosen to be 15  (See [Tr1]).
Hence, we only need to show that
\[
\int_{D}\Vert\left(\frac{F^{\star}H^{3}h}{FF^{\star}}-Q\,\widehat{\left(\frac{Q^{\star}F'^{\star}H^{3}h}{\left(FF^{\star}\right)^{2}}\right)}\right)^{'}\Vert^{2}\,dA(z)\leq C^2\Vert h\Vert_{\mathcal{D}}^{2}
\]
for some $C<\infty.$

\medskip
Now
\begin{align*}
\int_{D}& \Vert\left(\frac{F^{\star}H^{3}h}{FF^{\star}}-Q\widehat{\left(\frac{Q^{\star}F'^{\star}H^{3}h}{\left(FF^{\star}\right)^{2}}\right)}\right)^{'}\Vert^{2}\,dA(z)\\
& \quad \leq2\int_{D}\Vert\left(\frac{F^{\star}H^{3}h}{FF^{\star}}\right)^{'}\Vert^{2}\,dA(z)+2\int_{D}\Vert\left(Q\widehat{\left(\frac{Q^{\star}F'^{\star}H^{3}h}
{\left(FF^{\star}\right)^{2}}\right)}\right)^{'}\Vert^{2}\,dA(z)\\
& \quad \leq4\underset{(a')}{\underbrace{\int_{D}\Vert\frac{F^{\star}3H^{2}H'h}{FF^{\star}}\Vert^{2}\,dA(z)}}+8\underset{(b')}
{\underbrace{\int_{D}\Vert\frac{F^{\star}H^{3}h'}{FF^{\star}}\Vert^{2}\,dA(z)}}\\
& \qquad +8\underset{(c')}{\underbrace{\int_{D}\Vert\frac{F^{\star}H^{3}h'F'F^{\star}}
{\left(FF^{\star}\right)^{2}}\Vert^{2}\,dA(z)}}
+4\underset{(d')}{\underbrace{\int_{D}\Vert Q'\widehat{\left(\frac{Q^{\star}F'^{\star}H^{3}h}{\left(FF^{\star}\right)^{2}}\right)}
\Vert^{2}\,dA(z)}}\\
& \qquad +4\underset{(e')}{\underbrace{\int_{D}\Vert Q\left(\widehat{\frac{Q^{\star}F'^{\star}H^{3}h}{\left(FF^{\star}\right)^{2}}}\right)^{'}\Vert^{2}\,dA(z)}}.
\end{align*}

Then
\begin{align*}
(a')=\int_{D}\Vert\frac{F^{\star}3H^{2}H'h}{FF^{\star}}\Vert^{2}\,dA(z)& =9\int_{D}\Vert\frac{F^{\star}}{\sqrt{FF^{\star}}}\frac{H}{\sqrt{FF^{\star}}}H\, H\,'h\Vert^{2}\,dA(z)\\
& \leq9\int_{D}\Vert H\,'h\Vert^{2}\,dA(z)\\
& \leq18\left(\Vert M_{H}\Vert^{2}+\Vert H\Vert_{\infty}^{2}\right)\Vert h\Vert_{\mathcal{D}}^{2}\\
& \leq36\,\Vert M_{H}\Vert^{2}\,\Vert h\Vert_{\mathcal{D}}^{2}.
\end{align*}
\noindent
$\,(b')=\,\int_{D}\Vert\frac{F^{\star}H^{3}h'}{FF^{\star}}\Vert^{2}\,dA(z)\leq\int_{D}\Vert h'\Vert^{2}\,dA(z)\leq\Vert h\Vert_{\mathcal{D}}^{2}.$
\begin{align*}
(c')=\int_{D}\Vert\frac{F^{\star}H^{3}hF'F^{\star}}{\left(FF^{\star}\right)^{2}}\Vert^{2}\,dA(z)& =\int_{D}\Vert\frac{F^{\star}F'F^{\star}}{\sqrt{FF^{\star}}}\frac{H^{2}}{FF^{\star}}\frac{H}{\sqrt{FF^{\star}}}\,h\Vert^{2}\,dA(z)\\
& \leq\int_{D}\Vert\frac{F^{\star}F'F^{\star}}{\sqrt{FF^{\star}}}\,h\Vert^{2}\,dA(z)\\
& \leq\int_{D}\Vert F'^{\star}h\Vert^{2}\,dA(z)\leq4\,\Vert h\Vert_{\mathcal{D}}^{2}.
\end{align*}
 We use condition (a) of the theorem  and the boundedness of the Beurling transform on $L^2(\mathbb{D},dA)$ (with bound $14$) to conclude that
 \[
 (e')\leq56 ( {14)^2} \int_{D}\left|\vert F'^{\star}h\vert\right|^{2}dA\leq224(14)^2\Vert h\Vert_{\mathcal{D}}^{2}.
 \]
 So we only need estimate $(d')$.
For this, we have
\[
\int_{D}\Vert Q^{'}\,\,\widehat{\left(\frac{Q^{\star}F^{'\star}H^3h}{\left(FF^{\star}\right)^{2}}\right)}\Vert^{2}\,dA(z)=\int_{D}\Vert Q^{'}\widehat{w}\Vert^{2}\,dA(z),
\]
where $\widehat{w}=\widehat{\left(\frac{Q^{\star}F^{'\star}H^3h}{\left(FF^{\star}\right)^{2}}\right)}$
is a smooth function on $\overline{\mathbb{D}}$.

\medskip
Therefore,
\[
\int_{D}\Vert Q^{'}\,\widehat{w}\Vert^{2}\,dA(z)\leq2\underset{(\alpha)}{\underbrace{\int_{D}\Vert Q^{'}\widehat{w}-Q^{'}\widetilde{\widehat{w}}\Vert^{2}\,dA(z)}}+2\int_{D}\Vert Q^{'}\widetilde{\widehat{w}}\Vert^{2}\,dA(z),
\]
where $\widetilde{\widehat{w}}(z)=\int_{-\pi}^{\pi}\frac{1-\left|z\right|^{2}}{\left|1-e^{-it}z\right|}\,\widehat{w}(e^{it})\, d\sigma(t)$
is the harmonic extension of $\widehat{w}$ from $\partial \mathbb{D}$ to $\mathbb{D}$.

\medskip
 Lemma (2) tells us that
 \[
 \int_{D}\Vert Q^{'}\widetilde{\widehat{w}}\Vert^{2}\,dA(z)\leq8\,\Vert\widetilde{\widehat{w}}\Vert_{\mathcal{HD}}^{2}.
 \]
 Also, a lemma of {[}Tr2{]} implies that
 \[
 \Vert\widetilde{\widehat{w}}\Vert_{\mathcal{HD}}^{2}\leq\Vert w\Vert_{A}^{2}+\Vert\widehat{w}\Vert_{\sigma}^{2}.
 \]
But, as we showed above
\[
\Vert w\Vert_{A}^{2}=\int_{D}\Vert\frac{Q^{\star}F^{'\star}H^3h}{\left(FF^{\star}\right)^{2}}\Vert^{2}\,dA(z)\leq\int_{D}\Vert F^{'\star}h\Vert^{2}\,dA(z)
\leq4\,\Vert h\Vert_{\mathcal{D}}^{2}
\]
and
\[
\Vert\widehat{w}\Vert_{\sigma}^{2}=\int_{-\pi}^{\pi}\Vert\left(\widehat{\frac{Q^{\star}F^{'\star}H^3h}{\left(FF^{\star}\right)^{2}}}\right)\Vert^{2}\,d\sigma(t)\leq 15\left\Vert h\right\Vert _{\sigma}^{2}.
\]
from [Tr2].

\medskip
Thus,
\[
\int_{D}\Vert Q^{'}\widetilde{\widehat{w}}\Vert^{2}\,dA(z)\leq8\,\left[4\,\Vert h\Vert_{\mathcal{D}}^{2}+15\left\Vert h\right\Vert_{\sigma}^{2}\right].
\]

 Now we are just left with estimating $(\alpha)$. We will use Lemmas
3 and 4.
We have
\begin{align*}
(\alpha)& =\int_{D}\Vert Q^{'}\widehat{w}-Q^{'}\widetilde{\widehat{w}}\Vert^{2}\, dA(z)\\
& =\int_{D}\Vert Q^{'}\left[-\frac{1}{\pi}\int_{D}\frac{w(u)}{u-z}dA(u)-\int_{-\pi}^{\pi}\frac{1-\left|z\right|^{2}}{\left|1-e^{-it}z\right|}\widehat{w}(e^{it})d\sigma(t)\right]\Vert^{2}dA(z)\\
\end{align*}
\begin{multline*}
\qquad\, =\frac{1}{\pi^{2}}\int_{D}\Vert Q^{'}\int_{D}w(u)\bigg[\frac{1}{u-z}\\
 +\int_{-\pi}^{\pi}\frac{1-\left|z\right|^{2}}{\left|1-e^{-it}z\right|}\, e^{-it}\frac{1}{1-ue^{-it}}\,d\sigma(t)\bigg]dA(u)\Vert^{2}\,dA(z)
\end{multline*}
\begin{align*}
& =\frac{1}{\pi^{2}}\int_{D}\Vert Q^{'}\int_{D}w(u)\left[\frac{1}{u-z}+\frac{\bar{z}}{1-u\bar{z}}\right]dA(u)\Vert^{2}\,dA(z)\\
& =\frac{1}{\pi^{2}}\int_{D}\Vert Q'\int_{D}w(u)\left[\frac{1-\left|z\right|^{2}}{(u-z)(1-u\bar{z\,})}\right]dA(u)\Vert^{2}\,dA(z)\\
& =\frac{1}{\pi^{2}}\int_{D}\Vert Q'(z)\,(1-\vert z\vert^{2})\,T(w)(z)\Vert^2 \,dA(z)\\
& \leq\frac{\Vert M_{Q}\Vert^{2}}{\pi^{2}} \Vert T(w)\Vert^{2}_{A}  \;\text{   by Lemma 4  }\\
& \leq100\,\pi^{2}\,\frac{\Vert M_{Q}\Vert^{2}}{\pi^{2}}\,\Vert w\Vert^2 \;
\text{   by Lemma 3} \\
& \leq100\,\Vert M_{Q}\Vert^{2}\,\Vert h\Vert_{\mathcal{D}}^{2} \leq 1800 \Vert h\Vert_{\mathcal{D}}^{2}.
\end{align*}

\medskip \noindent
Combining all these pieces, we see that in the smooth case
\[
\left\Vert \underline{u}_{h}\right\Vert_{\mathcal{D}}^{2}\leq K\left\Vert h\right\Vert _{\mathcal{D}}^{2}\;
\]
for some constant $K< \infty$, which is independent of $h$ and $\epsilon >0$.

\medskip
 By the proof of Theorem 1 in the smooth case, we have
\[
 M_{F_{r}}^{R}(M_{F_{r}}^{R})^{\star}\leq K^{2}M_{H_{r}}M_{H_{r}}^{\star} \;
\text {for } 0\leq r<1.
\]
 Using a commutant lifting argument, there exists $G_{r}\in\mathcal{M}(\mathcal{D},\,\underset{1}{\overset{\infty}{\oplus}}\,\mathcal{D}$)
so that $M_{F_{r}}^{R}M_{G_{r}}^{C}=M_{H_{r}^3}$ and $\Vert M_{G_{r}}^{R}\Vert\leq K$.
Then $M_{F_{r}}^{R}\rightarrow M_{F}^{R}$ and $M_{H_{r}}\rightarrow M_{H}$
as $r\uparrow1$ in the $\star-$strong topology.

\medskip
 By compactness, we may choose a net with $G_{r_{\alpha}}^{\star}\rightarrow G^{\star}$
as $r_{\alpha}\rightarrow1^{-}.$ Since the multiplier algebra (as
operators) is WOT closed, $G\in\mathcal{M}(\mathcal{D},\overset{\infty}{\underset{1}{\oplus}}\,\mathcal{D}).$
Also, since $F_{r_{\alpha}}^{\star}\overset{s}{\rightarrow}F^{\star},$
we get $M_{H_{r}}^{\star}=M_{G_{r}}^{\star C}M_{F_{r}}^{\star R}\,\,\overset{WOT}{\rightarrow}\, M_{G}^{\star C}M_{F}^{\star R}$
and so $M_{F}^{R}M_{G}^{C}=M_{H^3}$ with entries of $G$ in $\mathcal{M}(\mathcal{D})$
and $\Vert M_{G}^{C}\Vert\leq K.$

\medskip

It might be of some interest to note that the norm of the operator, $\Vert  M_G^{C} \Vert$,
doesn't exceed $\sqrt{144 \Vert M_H \Vert^2 + 73,104}$.

\medskip 
This ends our proof.
\end{proof}

\bigskip
\newpage

Just as Wolff gets Theorem B for free, we get

\medskip
\begin{theorem}
Let $\{H, f_j: \, j=1, \dots, n\} \subset \mathcal{M}(\mathcal{D})$.  Then \, $H \in
Rad \, ( \{ f_j\}_{j=1}^n)$ \, if and only if there exist $C_0 < \infty$ and $ m \in \mathbb{N}$ such that
\[
|H^m(z)| \le C_0  \sum_{j=1}^n \, |f_j(z)|^2 \; \text{ for all } z \in \mathbb{D}.
\]
\end{theorem}

\bigskip
This paper discusses when  $H^3$ belongs to $\mathcal{I} ( \{ f_j\}_{j=1}^n)$, the ideal generated by $\{ f_j\}_{j=1}^n$ in $\mathcal{M}(\mathcal{D})$ and characterizes membership in the radical of the ideal,  $\mathcal{I} ( \{ f_j\}_{j=1}^n)$. The question of strong sufficient conditions for $H$ itself to belong to $\mathcal{I} ( \{ f_j\}_{j=1}^n)$ is more subtle. The first author has obtained some interesting results in this direction.

\end{document}